\documentclass[12pt,oneside,reqno]{amsart}
\usepackage[all]{xy}
\usepackage{amsfonts,amsmath,oldgerm,amssymb,amscd, comment,multirow}
\UseComputerModernTips
\numberwithin{equation}{section}
\usepackage[breaklinks]{hyperref}
\allowdisplaybreaks

\usepackage{enumerate}

\usepackage{caption} 
\newtheorem{theorem}{Theorem}[section]

\newtheorem{lemma}[theorem]{{\bf Lemma}}
\newtheorem{coro}[theorem]{{\bf Corollary}}
\newtheorem{definition}[theorem]{Definition}

\newtheorem{remark}[subsection]{Remark}

\newcommand{\al}{\alpha}

\begin{document}

	\title[ Arithmetic density and congruences of $\ell$-regular bipartitions $II$]{ Arithmetic density and congruences of $\ell$-regular bipartitions $II$} 
	
	\author[N.K. Meher*]{N.K. Meher}
	\address{Nabin Kumar Meher, Department of Mathematics, Indian Institute of Information Raichur, Govt. Engineering College Campus, Yermarus, Raichur, Karnataka, India 584135.}
	\email{mehernabin@gmail.com, nabinmeher@iiitr.ac.in}
	

	\thanks{
		$*$ is the corresponding author.}
	\maketitle
	\pagenumbering{arabic}
	\pagestyle{headings}
	\begin{abstract}
		Let $ B_{\ell}(n)$ denote the number of $\ell-$regular bipartitions of $n.$ In 2013, Lin \cite{Lin2013} proved a density result for $B_4(n).$ He showed that for any positive integer $k,$ $B_4(n)$ is almost always divisible by $2^k.$ In this article, we improved his result. We prove that $B_{2^{\alpha}m}(n)$  and $B_{3^{\alpha}m}(n)$ are almost always divisible by arbitrary power of $2$ and $3$ respectively. Further, we obtain an infinities families of congruences and multiplicative formulae for $B_2(n)$ and $B_4(n)$ by using Hecke eigenform theory. Next, by using a result of Ono and Taguchi on nilpotency of Hecke operator, we also find an infinite families of congruences modulo arbitrary power of $2$ satisfied by $B_{2^{\alpha}}(n).$
	\end{abstract}
	\maketitle
	
	\section{Introduction}
	For any positive integer $\ell >1,$ a partition is called $\ell$-regular if none of its parts is divisible by $\ell.$ Let $b_{\ell}(n)$ denotes the number of $\ell-$regular partitions of $n.$ The generating function for $b_{\ell}(n)$ is given by 
	$$ \sum_{n=0}^{\infty} b_{\ell}(n) q^n= \frac{f_{\ell}}{f_1},$$
	where $f_{\ell}$ is defined by $f_{\ell}= \prod_{m=1}^{\infty} (1- q^{\ell m}).$
	
	Gordan and Ono \cite{Gordan1997} proved that if $p$ is a prime number and $p^{\hbox{ord}_p(\ell)} \geq \sqrt{\ell},$ then for any positive integer $j,$ the arithmetic density of positive integers $n$ such that $b_{\ell}(n) \equiv 0 \pmod{p^j}$ is one. A bipartition $(\lambda, \mu)$ of $n$ is a pair of partitions $(\lambda, \mu)$ such that the sum of all of the parts is $n$.
	A $(k,\ell)$-regular bipartition of $ n$ is a bipartition $(\lambda, \mu)$ of $n$ such that $ \lambda$ is a $k$-regular partition and $\mu$ is an $\ell$-regular partition. An $\ell$-regular bipartition of $n$ is an ordered pair of $\ell$-regular partitions $(\lambda, \mu)$ such that the sum of all of the parts equals $n.$ Let $B_{\ell}(n)$ denote the number of $\ell$-regular bipartition of $n.$ Then the generating function of $B_{\ell}(n)$ satisfies 
	
	\begin{align}\label{eq501}
		\sum_{n=0}^{\infty} B_{\ell}(n) q^n= \frac{f^2_{\ell}}{f_1^2}.
	\end{align}
	
		Lin \cite{Lin2013} proved that $B_4(n)$  is even unless $n$ is of the form $k(k+1)$ for some $k \geq 0.$ He also showed that  $B_4(n)$ is a multiple of $4$ if $n$ is not the sum of two triangular numbers. Further,in this paper, he proved that, for any positive integer $k,$ $B_4(n)$ is almost always divisible by $2^k. $ In particular, he showed that $$\left\{ n \in \mathbb{N} : B_{4}(n)\equiv 0 \pmod{2^k} \right\}$$ has arithmetic density $1.$ In 2015, Dai \cite{Dai2015} proved several infinite families of congruences modulo $8$ for $B_4(n).$
	
     Motivated from the above results, in this article, we study the arithmetic density result of $ B_{2^{\alpha }m}(n)$ and $ B_{3^{\alpha }m}(n)$ similar to the result of Lin. Further, we obtain infinite families of congruences and multiplicate formulae for $B_{2}(n)$ and $B_4(n)$ using modular form techniques, Hecke eigenform. Furthermore, by using a result of Ono and Taguchi on nilpotency of Hecke operator, we also find an infinite families of congruences modulo arbitrary power of $2$ satisfied by $B_{2^{\alpha}}(n).$
     
	\begin{theorem}\label{mainthm16} Let $j$ be a fixed positive integer. Let $\alpha$ be a non negative integer and $m$ be a positive odd integer satisfying $2^{\alpha} \geq 2m.$ Then the set 
	$$\left\{ n \in \mathbb{N} : B_{2^{\alpha }m}(n)\equiv 0 \pmod{2^j} \right\}$$ has arithmetic density $1.$
    \end{theorem}

	\begin{theorem}\label{mainthm17} Let $j$ be a fixed positive integer. Let $\alpha$ be a non negative integer and $m$ be a positive odd integer co-prime to $3$ satisfying $3^{\alpha} \geq m.$ Then the set 
		$$\left\{ n \in \mathbb{N} : B_{3^{\alpha }m}(n)\equiv 0 \pmod{3^j} \right\}$$ has arithmetic density $1.$
	\end{theorem}	

	Serre observed and Tate showed that the action of Hecke operators on space of modular forms of level $1$ modulo $2$ is locally nilpotent (see \cite{Serre1974,Serre1974a,Tate1994}). Ono and Taguchi \cite{Taguchi2005} later generalized their result to higher levels. Using this method, Aricheta \cite{Aricheta2017}, found congruences for Andrew's singular overpartitions $ \overline{C}_{4k,k}(n)$ modulo $2$, when $k\in\{1,2,3\}.$ Recently, Singh and Barman \cite{Barman2021} applied this same method to Andrew's singular overpartitions $\overline{C}_{4.2^{\alpha}m, 2^{\alpha}m}(n),$  when $m=1$ and $m=3.$ We use this method here. We obtain that the eta-quotients associated to $\ell$-regular bipartitions $B_{\ell}(n)$ for $\ell = 2^\alpha m$ are modular forms whose level matches in the list of Ono and Taguchi. In fact, we prove the following result.
\begin{theorem}\label{mainthm3} There exists an integer $c \geq 0$ such that for every integer $d \geq 1$ and distinct primes $ q_1, q_2, \ldots, q_{c+d}$ coprime to $6,$ we have
	$$ B_{2^{\alpha}} \left(\frac{p_1p_2\cdots p_{c+d} \cdot n +1 -2^{\alpha +1}}{24} \right) \equiv 0 \pmod{2^{d}} $$
	whenever $n$ is coprime to $ q_1, q_2, \ldots, q_{c+d}.$
\end{theorem}

\begin{theorem}\label{thm12}
	Let $k$  be non-negative integer and $n$ be a positive integer. For each $1 \leq i \leq k+1,$ let $p_1, p_2,\ldots,p_{k+1}$ be primes such that $p_i  \not \equiv 1 \pmod {12}$. Then for any integer $j \not \equiv 0 \pmod {p_{k+1}},$ we have
	$$B_{2} \left( p_1^2 p_2^2 \cdots p_k^2 p_{k+1}^2 n + \frac{p_1^2 p_2^2 \cdots p_k^2 p_{k+1}\left(12j+ p_{k+1}\right)-1}{12} \right) \equiv 0 \pmod4.$$
\end{theorem}
If we put $p_1= p_2= \cdots= p_{k+1}= p$ in Theorem \ref{thm12}  then we obtain the result. 
\begin{coro}\label{coro7}
	Let $k$  be non-negative integer and $n$ be a positive integer. Let $p$ be a prime such that $p  \not \equiv 1 \pmod {12}.$ Then we have
	$$B_{2}\left(p^{2k+2}n +p^{2k+1}j+ \frac{p^{2k+2}-1}{12} \right) \equiv 0 \pmod 4,$$  whenever $ j \not \equiv 0 \pmod p.$
\end{coro}
In the above corollary, replacing $k$ by $k-1$ and putting $p=2, j=1$ we get
$$B_{2}\left(2^{2k}n +2^{2k-1}+ \frac{2^{2k}-1}{12} \right) \equiv B_{2}\left( 4^{k}n +\frac{7 \cdot 4^{k}-2}{12} \right)\equiv  0 \pmod 4.$$

	Furthermore, we prove the following multiplicative formulae for $2$-regular bipartition modulo $4$.
\begin{theorem}\label{thm13}
	Let $t \in \{ 5, 7, 11\}$. Let $k$ be a positive integer and $p$ be a prime number such that $ p \equiv t \pmod {12}.$ Let $r$ be a non-negative integer such that $p$ divides $12r+t,$ then
	\begin{align*}
		B_{2}\left(p^{k+1}n+pr+ \frac{tp-1}{12}\right)\equiv  B_{2}\left(p^{k-1}n+ \frac{12r+t-p}{12p} \right) \pmod4.
	\end{align*}
\end{theorem}
\begin{coro}\label{coro8}
	Let $t \in \{ 5, 7, 11\}$. Let $k$ be a positive integer and $p$ be a prime number such that $p \equiv t \pmod {12}.$ Then
	\begin{align*}
		B_{2}\left(p^{2k}n+ \frac{p^{2k}-1}{12} \right)&\equiv  B_{2}(n) \pmod4.
	\end{align*}
\end{coro}
In particular, if we put $p=5,$ in the above expression we obtain
\begin{align*}
	B_{2}\left(25^{k}n+ \frac{25^{k}-1}{12} \right)&\equiv    B_{2}(n)  \pmod4.
\end{align*}

 \begin{theorem}\label{thm14}
	Let $k$ be a non-negative integer and $n$ be a positive integer. For each $1 \leq i \leq k+1,$ let $p_1, p_2,\ldots,p_{k+1}$ be primes such that $p_i \equiv 3 \pmod4$. Then for any integer $j \not \equiv 0 \pmod {p_{k+1}},$ we have
	$$B_{4} \left( p_1^2 p_2^2 \cdots p_k^2 p_{k+1}^2 n + \frac{p_1^2 p_2^2 \cdots p_k^2 p_{k+1}\left(4j+ p_{k+1}\right)-1}{4} \right) \equiv 0 \pmod 4.$$
\end{theorem}
If we put $p_1= p_2= \cdots= p_{k+1}= p$ in Theorem \ref{thm14}  then we obtain the result. 
\begin{coro}\label{coro9}
	Let $k$ be a non-negative integer and $n$ be a positive integer.. Let $p$ be a prime such that $p \equiv 3 \pmod4.$ Then we have
	$$B_{4}\left(p^{2k+2}n +p^{2k+1}j+ \frac{p^{2k+2}-1}{4} \right) \equiv 0 \pmod 4$$ whenever $ j \not \equiv 0 \pmod p.$
\end{coro}
In the above corollary, if we put $p=3, j=1$ we get
$$B_{4}\left(3^{2k+2}n +\frac{7 \cdot 3^{2k}-1}{4} \right) \equiv  0 \pmod 4$$ which was earlier observed by H. Dai in \cite{Dai2015}.

Furthermore, we prove the following multiplicative formulae for $4$-regular bipartition modulo $4$.
\begin{theorem}\label{thm15}
	Let $k$ be a positive integer and $p$ be a prime number such that $ p \equiv 3 \pmod 4.$ Let $r$ be a non-negative integer such that $p$ divides $4r+3,$ then
	\begin{align*}
		B_{4}\left(p^{k+1}n+pr+ \frac{3p-1}{4}\right)\equiv  (-p^2) \cdot   B_{4}\left(p^{k-1}n+ \frac{4r+3-p}{4p} \right) \pmod4.
	\end{align*}
\end{theorem}
\begin{coro}\label{coro10}
	Let $k$ be a positive integer and $p$ be a prime number such that $p \equiv 3\pmod 4.$ Then
	\begin{align*}
		B_{4}\left(p^{2k}n+ \frac{p^{2k}-1}{4} \right)&\equiv  \left(-p^2\right)^k \cdot  B_{4}(n) \pmod4.
	\end{align*}
\end{coro}
In particular if we put $p=3,$ in the above expression we get 
\begin{align*}
	B_{4}\left(9^{k}n+ \frac{9^{k}-1}{4} \right)&\equiv  \left(-3^2\right)^k \cdot  B_{4}(n) \equiv     B_{4}(n) \pmod4.
\end{align*}

	
	\section{Preliminaries}
	We recall some basic facts and definition on modular forms. For more details, one can see \cite{Koblitz}, \cite{Ono2004}. We start with some matrix groups. We define
	\begin{align*}
		\Gamma:=\mathrm{SL_2}(\mathbb{Z})= &\left\{ \begin{bmatrix}
			a && b \\c && d
		\end{bmatrix}: a, b, c, d \in \mathbb{Z}, ad-bc=1 \right\},\\
		\Gamma_{\infty}:= &\left\{\begin{bmatrix}
			1 &n\\ 0&1	\end{bmatrix}: n \in \mathbb{Z}\right\}.
	\end{align*}
	For a positive integer $N$, we define
	\begin{align*}
		\Gamma_{0}(N):=& \left\{ \begin{bmatrix}
			a && b \\c && d
		\end{bmatrix} \in \mathrm{SL_2}(\mathbb{Z}) : c\equiv0 \pmod N \right\},\\
		\Gamma_{1}(N):=& \left\{ \begin{bmatrix}
			a && b \\c && d
		\end{bmatrix} \in \Gamma_{0}(N) : a\equiv d  \equiv 1 \pmod N \right\}
	\end{align*}
	and 
	\begin{align*}
		\Gamma(N):= \left\{ \begin{bmatrix}
			a && b \\c && d
		\end{bmatrix} \in \mathrm{SL_2}(\mathbb{Z}) : a\equiv d  \equiv 1 \pmod N,  b \equiv c  \equiv 0 \pmod N \right\}.
	\end{align*}
	A subgroup of $\Gamma=\mathrm{SL_2}(\mathbb{Z})$ is called a congruence subgroup if it contains $ \Gamma(N)$ for some $N$ and the smallest $N$ with this property is called its level. Note that $ \Gamma_{0}(N)$ and $ \Gamma_{1}(N)$ are congruence subgroup of level $N,$ whereas $ \mathrm{SL_2}(\mathbb{Z}) $ and $\Gamma_{\infty}$ are congruence subgroups of level $1.$ The index of $\Gamma_0(N)$ in $\Gamma$ is 
	\begin{align*}
		[\Gamma:\Gamma_0(N)]=N\prod\limits_{p|N}\left(1+\frac 1p\right)
	\end{align*}
	where $p$ runs over the prime divisors of $N$.
	
	Let $\mathbb{H}$ denote the upper half of the complex plane $\mathbb{C}$. The group 
	\begin{align*}
		\mathrm{GL_2^{+}}(\mathbb{R}):= \left\{ \begin{bmatrix}
			a && b \\c && d
		\end{bmatrix}: a, b, c, d \in \mathbb{R}, ad-bc>0 \right\},
	\end{align*}
	acts on $\mathbb{H}$ by $ \begin{bmatrix}
		a && b \\c && d
	\end{bmatrix} z = \frac{az+b}{cz+d}.$ We identify $\infty$ with $\frac{1}{0}$ and define $ \begin{bmatrix}
		a && b \\c && d
	\end{bmatrix} \frac{r}{s} = \frac{ar+bs}{cr+ds},$ where $\frac{r}{s} \in \mathbb{Q} \cup \{ \infty\}$. This gives an action of $\mathrm{GL_2^{+}}(\mathbb{R})$ on the extended half plane $\mathbb{H}^{*}=\mathbb{H} \cup \mathbb{Q} \cup \{\infty\}$. Suppose that $\Gamma$ is a congruence subgroup of $\mathrm{SL_2}(\mathbb{Z})$. A cusp of $\Gamma$ is an equivalence class in $\mathbb{P}^{1}=\mathbb{Q} \cup \{\infty\}$ under the action of $\Gamma$.
	
	The group $\mathrm{GL_2^{+}}(\mathbb{R})$ also acts on functions $g:\mathbb{H} \rightarrow \mathbb{C}$. In particular, suppose that $\gamma=\begin{bmatrix}
		a && b \\c && d
	\end{bmatrix}\in \mathrm{GL_2^{+}}(\mathbb{R})$. If $f(z)$ is a meromorphic function on $\mathbb{H}$ and $k$ is an integer, then define the slash operator $|_{k}$ by
	\begin{align*}
		(f|_{k} \gamma)(z):= (\det \gamma)^{k/2} (cz+d)^{-k} f(\gamma z).
	\end{align*}
	
	\begin{definition}
		Let $\Gamma$ be a congruence subgroup of level $N$. A holomorphic function $f:\mathbb{H} \rightarrow \mathbb{C}$ is called a modular form with integer weight $k$ on $\Gamma$ if the following hold:
		\begin{enumerate}[$(1)$]
			\item We have
			\begin{align*}
				f \left( \frac{az+b}{cz+d}\right)=(cz+d)^{k} f(z)
			\end{align*}
			for all $z \in \mathbb{H}$ and $\begin{bmatrix}
				a && b \\c && d
			\end{bmatrix}\in \Gamma$. 
			\item If $\gamma\in SL_2 (\mathbb{Z})$, then $(f|_{k} \gamma)(z)$ has a Fourier expnasion of the form
			\begin{align*}
				(f|_{k} \gamma)(z):= \sum \limits_{n\geq 0}a_{\gamma}(n) q_N^{n}
			\end{align*}
			where $q_N:=e^{2\pi i z /N}$.
		\end{enumerate}
	\end{definition}
	For a positive integer $k$, the complex vector space of modular forms of weight $k$ with respect to a congruence subgroup $\Gamma$ is denoted by $M_{k}(\Gamma)$.
	
	\begin{definition} \cite[Definition 1.15]{Ono2004}
		If $\chi$ is a Dirichlet character modulo $N$, then we say that a modular form $g \in M_{k}(\Gamma_1(N))$ has Nebentypus character $\chi$ if 
		\begin{align*}
			f \left( \frac{az+b}{cz+d}\right)=\chi(d) (cz+d)^{k} f(z)
		\end{align*}
		for all $z \in \mathbb{H}$ and $\begin{bmatrix}
			a && b \\c && d
		\end{bmatrix}\in \Gamma_{0}(N)$. The space of such modular forms is denoted by $M_{k}(\Gamma_0(N), \chi)$.
	\end{definition}
	
	The relevant modular forms for the results obtained in this article arise from eta-quotients. Recall that the Dedekind eta-function $\eta (z)$ is defined by 
	\begin{align*}
		\eta (z):= q^{1/24}(q;q)_{\infty}=q^{1/24} \prod\limits_{n=1}^{\infty} (1-q^n)
	\end{align*}
	where $q:=e^{2\pi i z}$ and $z \in \mathbb{H}$. A function $f(z)$ is called an eta-quotient if it is of the form
	\begin{align*}
		f(z):= \prod\limits_{\delta|N} \eta(\delta z)^{r_{\delta}}
	\end{align*}
	where $N$ and $r_{\delta}$ are integers with $N>0$. 
	
	\begin{theorem} \cite[Theorem 1.64]{Ono2004} \label{thm2.3}
		If $f(z)=\prod\limits_{\delta|N} \eta(\delta z)^{r_{\delta}}$ is an eta-quotient such that $k= \frac 12$ $\sum_{\delta|N} r_{\delta}\in \mathbb{Z}$, 
		\begin{align*}
			\sum\limits_{\delta|N} \delta r_{\delta} \equiv 0\pmod {24}	\quad \textrm{and} \quad \sum\limits_{\delta|N} \frac{N}{\delta}r_{\delta} \equiv 0\pmod {24},
		\end{align*}
		then $f(z)$ satisfies
		\begin{align*}
			f \left( \frac{az+b}{cz+d}\right)=\chi(d) (cz+d)^{k} f(z)
		\end{align*}
		for each $\begin{bmatrix}
			a && b \\c && d
		\end{bmatrix}\in \Gamma_{0}(N)$. Here the character $\chi$ is defined by $\chi(d):= \left(\frac{(-1)^{k}s}{d}\right)$ where $s=\prod_{\delta|N} \delta ^{r_{\delta}}$.
	\end{theorem}
	
	\begin{theorem} \cite[Theorem 1.65]{Ono2004} \label{thm2.4}
		Let $c,d$ and $N$ be positive integers with $d|N$ and $\gcd(c,d)=1$. If $f$ is an eta-quotient satisfying the conditions of Theorem \ref{thm2.3} for $N$, then the order of vanishing of $f(z)$ at the cusp $\frac{c}{d}$ is
		\begin{align*}
			\frac{N}{24}\sum\limits_{\delta|N} \frac{\gcd(d, \delta)^2 r_{\delta}}{\gcd(d, \frac{N}{ d} )d \delta}.
		\end{align*}
	\end{theorem}
	Suppose that $f(z)$ is an eta-quotient satisfying the conditions of Theorem \ref{thm2.3} and that the associated weight $k$ is a positive integer. If $f(z)$ is holomorphic at all of the cusps of $\Gamma_0(N)$, then $f(z) \in M_{k}(\Gamma_0(N), \chi)$. Theorem \ref{thm2.4} gives the necessary criterion for determining orders of an eta-quotient at cusps. In the proofs of our results, we use Theorems \ref{thm2.3} and \ref{thm2.4} to prove that $f(z) \in M_{k}(\Gamma_0(N), \chi)$ for certain eta-quotients $f(z)$ we consider in the sequel.
	
	We shall now mention a result of Serre \cite[P. 43]{Serre1974} which will be used later. 
	
	\begin{theorem}\label{thm2.5}
		Let $f(z) \in M_{k}(\Gamma_{0}(N), \chi)$ has Fourier expansion
		$$ f(z)= \sum_{n=0}^{\infty} b(n) q^n \in \mathbb{Z}[[q]].$$
		Then for a positive integer $r$,  there is a constant $\alpha>0$ such that
		$$\#\{ 0 < n \leq X: b(n) \not \equiv 0 \pmod{r}\} = \mathcal{O}\left( \frac{X}{(\log X)^{\alpha}}\right).$$
		Equivalently 
		\begin{align}\label{2e1}
			\begin{split}
				\lim\limits_{X \to \infty} \frac{\#\{ 0 < n \leq X: b(n) \not \equiv 0 \pmod{r}\}}{X}= 0.
		\end{split}	\end{align}
	\end{theorem}

		We finally recall the definition of Hecke operators and a few relavent results. Let $m$ be a positive integer and $f(z)= \sum \limits_{n= 0}^{\infty}a(n) q^{n}\in M_{k}(\Gamma_0(N), \chi)$. Then the action of Hecke operator $T_m$ on $f(z)$ is defined by
	\begin{align*}
		f(z)|T_{m} := \sum \limits_{n= 0}^{\infty} \left(\sum \limits_{d|\gcd(n,m)} \chi(d) d^{k-1} a\left(\frac{mn}{d^2}\right)\right)q^{n}.
	\end{align*}
	In particular, if $m=p$ is a prime, we have
	\begin{align*}
		f(z)|T_p := \sum \limits_{n= 0}^{\infty}\left( a(pn) + \chi(p) p^{k-1} a\left(\frac{n}{p}\right)\right)q^{n}.
	\end{align*}
	We note that $a(n)=0$ unless $n$ is a non-negative integer.

	\section{Proof of Theorem \ref{mainthm16}}
	Let $m=p_1^{a_1}p_2^{a_2}\ldots p_r^{a_r},$ where $p_i$'s
	are odd prime numbers .
	From \eqref{eq501}, we get
	\begin{small}
		\begin{equation}\label{eq601}
			\sum_{n=0}^{\infty} B_{2^{\alpha }m}(n) q^n =  \frac{(q^{2^{\alpha }m};q^{2^{\alpha }m})^{2}_{\infty}}{(q;q)^2_{\infty}}.
		\end{equation}
	\end{small}
	Note that for any prime $p$ and positive integers $j$  we have
	\begin{small}
		\begin{equation}\label{eq602}
			(q;q)_{\infty}^{p^j}\equiv (q^p;q^p)_{\infty}^{p^{j-1}} \pmod{p^j}.
		\end{equation}
	\end{small}
	For a positive integer $i,$ we define $$E_{\alpha, m}(z):= \frac{  \eta^2( 3 \cdot 2^{  \alpha +3} mz)}{  \eta( 3 \cdot 2^{  \alpha +4} mz)} .$$
	Using \eqref{eq602}, we get 
	\begin{small}
		\begin{equation*}
			E^{2^j}_{\alpha, m}(z)=  \frac{  \eta^{2^{j+1}}( 3 \cdot 2^{  \alpha +3} mz)}{  \eta^{2^j}( 3 \cdot 2^{  \alpha +4} mz)}  \equiv 1 \pmod {2^{j+1}}.
		\end{equation*}
	\end{small}
	Define 
	$$F_{\alpha, m,j}(z)= \frac{  \eta^2( 3 \cdot 2^{  \alpha +3} mz)}{  \eta^2 ( 3 \cdot 2^{3} z)} E^{2^j}_{\alpha, m}(z)= \frac{  \eta^{(2^{j+1}+2)}( 3 \cdot 2^{  \alpha +3} mz)}{  \eta^2 ( 3 \cdot 2^{3} z) \eta^{2^j}( 3 \cdot 2^{  \alpha +4} mz) } .$$
	On modulo $2^{j+1},$ we get
	\begin{small}
		\begin{equation}\label{eq603}
			F_{\alpha, m,j}(z) \equiv  \frac{  \eta^2( 3 \cdot 2^{  \alpha +3} mz)}{  \eta^2 ( 3 \cdot 2^{3} z)} = q^{(2^{\alpha +1}m -1)} \frac{(q^{ 3 \cdot 2^{\alpha +3 }m};q^{ 3 \cdot 2^{\alpha +3 }m})^{2}_{\infty}}{(q^{24};q^{24})^2_{\infty}}.
		\end{equation}
	\end{small}
	Combining \eqref{eq601} and \eqref{eq603} together, we obtain
	\begin{small}
		\begin{equation}\label{eq604a}
			F_{\alpha, m,j}(z)\equiv  q^{(2^{\alpha +1}m -1)} \frac{(q^{ 3 \cdot 2^{\alpha +3 }m};q^{ 3 \cdot 2^{\alpha +3 }m})^{2}_{\infty}}{(q^{24};q^{24})^2_{\infty}} \equiv \sum_{n=0}^{\infty} B_{2^{\alpha}m}(n)q^{24n+ (2^{\alpha +1}m -1)} \pmod{2^{j+1}}.
		\end{equation}
	\end{small}
	
	Next, we prove that $F_{\alpha, m,j}(z)$ is a modular form  for certain values of $\alpha, m , \ \hbox{and} \ j$. 
	
	\begin{lemma}\label{lem11}
		Let $\al$ be a non-negative integer and $m=p_1^{\al_1} p_2^{\al_2}\cdots p_r^{\al_r}, $ where $\al_i \geq 0$ and $p_i \geq 3$ be distinct primes. If $2^{\al} \geq 2m,$ we have $ F_{\alpha, m,j}(z) \in  M_{2^{j-1}}(\Gamma_{0}(N), \chi_1)$ for all $j \geq 2\alpha,$ where $N= 9 \cdot 2^{\alpha+6} \cdot m$ and $\chi_1= $. 
		
	\end{lemma}
	\begin{proof}
		Applying Theorem \ref{thm2.3}, we first estimate the level of eta quotient $F_{\alpha, m,j}(z)$ . The level of $ F_{\alpha, m,j}(z) $ is $N= 3 \cdot 2^{  \alpha +4} m \cdot M,$ where $M$ is the smallest positive integer which satisfies 
		\begin{small}
			\begin{align*}
				3 \cdot 2^{  \alpha +4} m \cdot M & \left[ \frac{2^{j+1}+2}{2^{\alpha +3} \cdot 3\cdot  m} - \frac{2}{2^3 \cdot 3} - \frac{2^{j}}{2^{\alpha +4} \cdot 3\cdot  m} \right]\equiv 0 \pmod{24} \\
				\implies 2M & \left[ 2 \cdot 2^j +2 - 2^{\alpha +1} \cdot m - 2^{j-1} \right]\equiv 0 \pmod{24}.
			\end{align*}
		\end{small}
		Therefore $M=12$ and the level of $F_{\alpha, m,j}(z)$  is $N=2^{\alpha +6} \cdot 3^2 \cdot   m $.
		The cusps of $\Gamma_{0}(2^{\alpha +6} \cdot 3^2 \cdot   m)$ are given by fractions $\frac{c}{d}$ where $d|2^{\alpha +6} \cdot 3^2 \cdot   m$ and $\gcd(c,d)=1.$ By using Theorem \ref{thm2.4}, we obtain that $F_{\alpha, m,j}(z)$ is holomorphic at a cusp $\frac{c}{d}$ if and only if 
		
		\begin{align}\label{eq605}
			&(2^{j+1}+2) \frac{  \gcd^2(d, 2^{\alpha +3}3 m)}{2^{\alpha +3}3 m} - 2 \frac{\gcd^2(d, 24)}{24} -2^j \frac{\gcd^2(d, 2^{\alpha +4}3 m )}{2^{\alpha +4}3 m} \geq 0 \\ \notag
			& \iff L:= \left( 4 \cdot  2^j+4\right)  G_1 -  \left(4 \cdot 2^{\alpha} \cdot m \right)  G_2 - 2^j \geq 0,
		\end{align}
		
		where $G_1= \frac{\gcd^2(d, 2^{\alpha +3} \cdot 3 \cdot  m)}{ \gcd^2(d, 2^{\alpha +4} \cdot 3 \cdot  m )}$ and  $G_2= \frac{\gcd^2(d, 24)}{\gcd^2(d, 2^{\alpha +4} \cdot 3 \cdot  m ) } .$ 
		
		Let $d$ be a divisor of $2^{\alpha +6} \cdot 3^2 \cdot   m$. We can write $d= 2^{r_1} 3^{r_2} t $ where $ 0 \leq r_1 \leq {\alpha +6}$,   $ 0 \leq r_2 \leq 2$, and $t|m $.

		Next, we find all possible value of divisors of $2^{\alpha +6} \cdot 3^2 \cdot   m$ to compute equation \eqref{eq605}.
		
		Case-$(i)$: When $d= 2^{r_1}3^{r_2}t : 0 \leq r_1\leq \alpha +3, 0 \leq r_2\leq 2, t|m.$ Then $G_1= 1, $ $   \frac{1}{ 2^{2 \alpha} t^2} \leq G_2 \leq 1.$
		Then equation \eqref{eq605} will be
		\begin{align}\label{eq607}
			L \geq  \left( 4 \cdot  2^j+4\right) \cdot 1 -  \left(4 \cdot 2^{\alpha} \cdot m \right) \cdot  1 - 2^j
			= 3 \cdot  2^j+4 - \left( 4 \cdot 2^{\alpha} \cdot m\right).   
		\end{align}
		
		Since $j \geq 2 \alpha $ and $2^{\alpha} \geq 2 m$, we have
		
		\begin{align*}
			L \geq   3 \cdot  2^j+4 - \left( 4 \cdot 2^{\alpha} \cdot m\right) \geq  2 \cdot  2^{2 \alpha}- \left( 4 \cdot 2^{\alpha} \cdot m\right)  +4 = 2 \cdot  2^{ \alpha} \left( 2^{ \alpha} - 2m \right) +4 >0.
		\end{align*}
		
		Case-$(ii)$: When $d= 2^{r_1}3^{r_2}t : \alpha +4 \leq r_1\leq \alpha +6, 0 \leq r_2\leq 2, t|m.$ Then $G_1= \frac{1}{4}, $ $   \frac{1}{ 2^{2 \alpha+2} t^2} \leq G_2 \leq \frac{1}{2^{2 \alpha +2}}.$
		Since $2^{\alpha} \geq 2 m,$ the equation \eqref{eq605} will be
		\begin{align}\label{eq608}
			L \geq  \left( 4 \cdot  2^j+4\right) \cdot \frac{1}{4} -  \left(4 \cdot 2^{\alpha} \cdot m \right) \cdot  \frac{1}{4. 2^{2 \alpha}} - 2^j
			= 1 - \left(  \frac{m}{2^{\alpha}} \right) \geq 1- \frac{1}{2} >0. 
		\end{align}
		Therefore, $F_{\alpha, m,j}(z)$ is holomorphic at every cusp $\frac{c}{d}.$
		Using Theorem \ref{thm2.3}, we compute the weight of $F_{\alpha, m,j}(z)$ is $k= 2^{j-1} $ which is a positive integer. The associated character for $F_{\alpha, m,j}(z)$ is $$\chi_1= \left( \frac{(-1)^{2^{j-1}}3^{2^{j}}2^{( \alpha+ 2).2^j+2 \alpha} m^{2^{j}+2}}{\bullet}\right).$$ 
		Thus $ F_{\alpha, m,j}(z) \in  M_{k}(\Gamma_{0}(N), \chi_1)$ where $k$, $N$ and $\chi_1$ are as above.
	\end{proof}
	\subparagraph*{Proof of Theorem \ref{mainthm16}}
	For a given fixed $\alpha \geq 0,$ it is sufficient to prove Theorem \ref{mainthm16} for all $j \geq 2 \alpha.$ Since $2^{\alpha} \geq 2m,$ by Lemma \ref{lem11}, we get $ F_{\alpha, m,j}(z) \in  M_{2^{j-1}}(\Gamma_{0}(N), \chi)$ where $N= 9 \cdot 2^{\alpha+6} \cdot m$ and $\chi_1= \left( \frac{(-1)^{2^{j-1}}3^{2^{j}}2^{( \alpha+ 2).2^j+2 \alpha} m^{2^{j}+2}}{\bullet}\right) $.

	Applying Theorem \ref{thm2.5}, we obtain that the Fourier coefficients of $F_{\alpha, m,j}(z) $ satisfies \eqref{2e1} for $r=2^j$ which implies that the Fourier coefficient of $F_{\alpha, m,j}(z) $ are almost always divisible by $ 2^j$. Hence, from $\eqref{eq604a}$, we conclude that $B_{\ell}(n)$ are almost always divisible by $2^j$. This completes the proof of Theorem \ref{mainthm16}.
	\qed 

		\section{Proof of Theorem \ref{mainthm17}}
	Let $m$ be a positive integer such that $3 \nmid m.$ 
	From \eqref{eq501}, we get
	\begin{small}
		\begin{equation}\label{eq651}
			\sum_{n=0}^{\infty} B_{3^{\alpha }m}(n) q^n =  \frac{(q^{3^{\alpha }m};q^{3^{\alpha }m})^{2}_{\infty}}{(q;q)^2_{\infty}}.
		\end{equation}
	\end{small}
	Note that for any prime $p$ and positive integers $j$  we have
	\begin{small}
		\begin{equation}\label{eq652}
			(q;q)_{\infty}^{p^j}\equiv (q^p;q^p)_{\infty}^{p^{j-1}} \pmod{p^j}.
		\end{equation}
	\end{small}
	Let $$G_{\alpha, m}(z):= \frac{  \eta^3( 2^3 \cdot 3^{  \alpha +1} mz)}{  \eta( 2^3 \cdot 3^{  \alpha +2} mz)} .$$
	Using the binomial theorem, we get 
	\begin{small}
		\begin{equation*}
			G^{3^j}_{\alpha, m}(z)=  \frac{  \eta^{3^{j+1}}( 2^3 \cdot 3^{  \alpha +1} mz)}{  \eta^{3^j}(2^3 \cdot 3^{  \alpha +2} mz)}  \equiv 1 \pmod {3^{j+1}}.
		\end{equation*}
	\end{small}
	Define 
	$$H_{\alpha, m,j}(z)= \frac{  \eta^2( 2^3 \cdot 3^{  \alpha +1} mz)}{  \eta^2 ( 3 \cdot 2^{3} z)} G^{3^j}_{\alpha, m}(z)= \frac{  \eta^{(3^{j+1}+2)}( 2^3 \cdot 3^{  \alpha +1} mz)}{  \eta^2 ( 3 \cdot 2^{3} z) \eta^{3^j}( 2^3 \cdot 3^{  \alpha +2} mz) } .$$
	On modulo $3^{j+1},$ we obtain
	\begin{small}
		\begin{equation}\label{eq653}
			H_{\alpha, m,j}(z) \equiv  \frac{  \eta^2( 2^3 \cdot 3^{  \alpha +1} mz)}{  \eta^2 ( 3 \cdot 2^{3} z)} = q^{(2\cdot 3^{\alpha}m -2)} \frac{(q^{ 2^3 \cdot 3^{\alpha +1 }m};q^{ 2^3 \cdot 3^{\alpha +1 }m})^{2}_{\infty}}{(q^{24};q^{24})^2_{\infty}}.
		\end{equation}
	\end{small}
	Combining \eqref{eq651} and \eqref{eq653} together, we obtain
	\begin{small}
		\begin{equation}\label{eq654a}
			H_{\alpha, m,j}(z)\equiv  q^{(2\cdot 3^{\alpha}m -2)} \frac{(q^{ 2^3 \cdot 3^{\alpha +1 }m};q^{ 2^3 \cdot 3^{\alpha +1 }m})^{2}_{\infty}}{(q^{24};q^{24})^2_{\infty}} \equiv \sum_{n=0}^{\infty} B_{3^{\alpha}m}(n)q^{24n+ (2\cdot 3^{\alpha}m -2)} \pmod{3^{j+1}}.
		\end{equation}
	\end{small}
	
	Next, we show that $H_{\alpha, m,j}(z)$ is a modular form  for certain values of $\alpha, m , \ \hbox{and} \ j$. 
	\begin{lemma}\label{lem12}
		Let $\al$ be a non-negative integer and $m$ be a positive integer such that $3 \nmid m$ . If $3^{\al} \geq m,$  then we have $ H_{\alpha, m,j}(z) \in M_{3^j}(\Gamma_{0}(N), \chi_2)$ for all $j \geq 2\alpha$, where $N= 2^6 \cdot 3^{ \alpha +2} m $ and $\chi_2= \left( \frac{\left((-1)\cdot 2^{63^j} \cdot 3^{(2\alpha+1)3^j+2 \alpha} \cdot m^{2.3^j+2}  \right)}{\bullet}\right)$ .
		
	\end{lemma}
	\begin{proof}
		Applying Theorem \ref{thm2.3}, we first estimate the level of eta quotient $H_{\alpha, m,j}(z)$ . The level of $ H_{\alpha, m,j}(z) $ is $N= 2^3 \cdot 3^{  \alpha +2} m \cdot M,$ where $M$ is the smallest positive integer which satisfies 
		\begin{small}
			\begin{align*}
				2^3 \cdot 3^{  \alpha +2} m \cdot M & \left[ \frac{3^{j+1}+2}{2^3 \cdot 3^{  \alpha +1} m } - \frac{2}{2^3 \cdot 3} - \frac{3^{j}}{2^3 \cdot 3^{  \alpha +2} m} \right]\equiv 0 \pmod{24} \\
				\implies 3M & \left[ 3 \cdot 3^j +2 - 2 \cdot 3^{\alpha} \cdot m - 3^{j-1} \right]\equiv 0 \pmod{24}.
			\end{align*}
		\end{small}
		Therefore, $M=8$ and the level of $H_{\alpha, m,j}(z)$  is $N=2^6 \cdot 3^{  \alpha +2} m $.
		The cusps of $\Gamma_{0}(2^6 \cdot 3^{  \alpha +2} m)$ are given by fractions $\frac{c}{d}$ where $d|2^6 \cdot 3^{  \alpha +2} m$ and $\gcd(c,d)=1.$ By using Theorem \ref{thm2.4}, we obtain that $H_{\alpha, m,j}(z)$ is holomorphic at a cusp $\frac{c}{d}$ if and only if 
		\begin{align}\label{eq655}
			&(3^{j+1}+2) \frac{  \gcd^2(d, 2^3 \cdot 3^{\alpha +1} m)}{2^3 \cdot 3^{\alpha +1} m} - 2 \frac{\gcd^2(d, 24)}{24} -3^j \frac{\gcd^2(d, 2^3 \cdot 3^{\alpha +2} m )}{2^3 \cdot 3^{\alpha +2} m} \geq 0 \\ \notag
			& \iff L:=3 \cdot  \left(3^{j+1}+2\right)  G_1 -  \left(2 \cdot 3^{\alpha+1} \cdot m \right)  G_2 - 3^j \geq 0,
		\end{align}
		
		where $G_1= \frac{\gcd^2(d, 2^3 \cdot 3^{\alpha +1} m)}{ \gcd^2(d, 2^3 \cdot 3^{\alpha +2} m )}$ and  $G_2= \frac{\gcd^2(d, 24)}{\gcd^2(d, 2^3 \cdot 3^{\alpha +2} m )  } .$ 
		Let $d$ be a divisor of $2^6 \cdot 3^{  \alpha +2} m$. We can write $d= 2^{r_1} 3^{r_2} t $ where $ 0 \leq r_1 \leq 6$,   $ 0 \leq r_2 \leq  \alpha +2 $, and $t|m $.
		
		Next, we find all possible value of divisors of $2^6 \cdot 3^{  \alpha +2} m$ to compute equation \eqref{eq655}.
		Case-$(i)$: When $d= 2^{r_1}3^{r_2}t : 0 \leq r_1\leq 6, 0 \leq r_2\leq \alpha+1, t|m.$ Then $G_1= 1, $ $   \frac{1}{ 3^{2 \alpha} t^2} \leq G_2 \leq \frac{1}{3^{2 \alpha}}.$
		Then equation \eqref{eq655} will be
		\begin{align}\label{eq657}
			L \geq 3 \cdot  \left(3^{j+1}+2\right) \cdot 1 -  \left(2 \cdot 3^{\alpha+1} \cdot m \right) \frac{1}{3^{2 \alpha}}  - 3^j
			= 8 \cdot  3^j+6 - \frac{6m}{3^{\alpha}}.    
		\end{align}
		Since $j \geq 2 \alpha $ and $3^{\alpha} \geq  m$, we have
		\begin{align*}
			L \geq   8 \cdot  3^j+6 - \frac{6m}{3^{\alpha}} \geq  8 \cdot  3^{2 \alpha}- \frac{6m}{3^{\alpha}}  + 6 >0.
		\end{align*}
		Case-$(ii)$: When $d= 2^{r_1}3^{r_2}t : 0 \leq  r_1\leq 6, \alpha +1 \leq r_2\leq \alpha+2, t|m.$ Then $G_1= \frac{1}{9}, $ $   \frac{1}{ 3^{2 \alpha+2} t^2} \leq G_2 \leq \frac{1}{3^{2 \alpha +2}}.$
		Since $3^{\alpha} \geq m,$ the equation \eqref{eq655} will be
		\begin{align}\label{eq658}
			L \geq  3 \cdot  \left(3^{j+1}+2\right) \cdot \frac{1}{9} -  \left(2 \cdot 3^{\alpha+1} \cdot m \right) \frac{1}{3^{2 \alpha +2}} - 3^j
			= \frac{2}{3} - \frac{2}{3} \cdot  \left(  \frac{m}{3^{\alpha}} \right) \geq 0. 
		\end{align}
		Therefore, $H_{\alpha, m,j}(z)$ is holomorphic at every cusp $\frac{c}{d}.$
		Using Theorem \ref{thm2.3}, we compute the weight of $H_{\alpha, m,j}(z)$ is $k= 3^{j} $ which is a positive integer. The associated character for $H_{\alpha, m,j}(z)$ is $$\chi_2= \left( \frac{\left((-1)\cdot 2^{63^j} \cdot 3^{(2\alpha+1)3^j+2 \alpha} \cdot m^{2.3^j+2}  \right)}{\bullet}\right).$$ 
		Thus,  $ H_{\alpha, m,j}(z) \in  M_{3^j}(\Gamma_{0}(N), \chi_2)$ where $k$, $N$ and $\chi_2$ are as above.
	\end{proof}
	\subparagraph*{Proof of Theorem \ref{mainthm17}}
	For a given fixed $\alpha \geq 0,$ it is sufficient to prove Theorem \ref{mainthm17} for all $j \geq 2 \alpha.$ Since $3^{\alpha} \geq m,$ by Lemma \ref{lem12}, we get $ H_{\alpha, m,j}(z) \in  M_{3^{j}}(\Gamma_{0}(N), \chi)$ where $N= 2^6 \cdot 3^{  \alpha +2} m $ and $\chi_2= \left( \frac{\left((-1)\cdot 2^{63^j} \cdot 3^{(2\alpha+1)3^j+2 \alpha} \cdot m^{2.3^j+2}  \right)}{\bullet}\right) $.

	Applying Theorem \ref{thm2.5}, we obtain that the Fourier coefficients of $H_{\alpha, m,j}(z) $ satisfies \eqref{2e1} for $r=3^j$ which implies that the Fourier coefficient of $H_{\alpha, m,j}(z) $ are almost always divisible by $ 3^j$. Hence, from $\eqref{eq654a}$, we conclude that $B_{\ell}(n)$ are almost always divisible by $2^j$. This completes the proof of Theorem \ref{mainthm17}.
	\qed
	%

	\section{Proof of Theorem \ref{mainthm3} }
	In this section, we prove Theorem \ref{mainthm3} by using a deep result of Ono and Taguchi on nilpotency of Hecke operators. The following theorem is a generalization of the result of Ono and Taguchi \cite[Theorem $1.3$]{Ono2004}, which was used by Archita, Ajit and Rupam (see \cite{Aricheta2017,Barman2021,Ajit2021}.) 
	\begin{theorem}\label{thm2.41}
		Let $n$ be a non-negative integer and $k$ be a positive integer. Let $\chi$ be a quadratic Dirichlet character of conductor $9\cdot2^n$. Then there is an integer $r\geq 0$ such that for every $h(z) \in M_{k}(\Gamma_{0}(9\cdot2^a), \chi) \cap \mathbb{Z}[[q]]$ and every $u \geq 1$
		$$ g(z)|T_{p_1}|T_{p_2}\cdots |T_{p_{r+s}} \equiv 0\pmod{2^u}$$ 
		whenever the primes $p_1, p_2, \ldots, p_{r+s} $ are co-prime to $6.$
	\end{theorem}
	\begin{remark}
		In  \cite[Theorem $1.3$]{Ono2004}, Ono and Taguchi stated the above result for the space of cusps forms; however they put a remark after the theorem, which ensure that we can use their result for modular forms. This result has already been used for modular forms by the authors in \cite{Aricheta2017,Barman2021,Ajit2021}. 
	\end{remark}
	\begin{remark}
		Note that the Kronecker symbol $\left(\frac{a}{\bullet} \right)$ is a Dirichlet character if and only if $a \not \equiv 3 \pmod 4$ (see \cite{Allouche2018}.)
	\end{remark}
	{\bf Proof of Theorem \ref{mainthm3}:}
	Let $j$ be a fixed positive integer. From \eqref{eq604a} for $m=1$ and all $\alpha \geq 0$, we have
	\begin{equation*}
		S_{\alpha}(z) = F_{\alpha, 1,j}(z) \equiv \sum_{n=0}^{\infty} B_{2^{\alpha}}(n)q^{24n+ (2^{\alpha +1} -1)} \pmod{2^{j+1}}. 
	\end{equation*}
	This implies that
	\begin{equation}\label{5e1}
		S_{\alpha}(z):= \sum_{n=0}^{\infty} C_{\alpha}(n) q^n \equiv \sum_{n=0}^{\infty} B_{2^{\alpha}} \left(\frac{n+1-2^{\alpha +1}}{24}\right)  q^n \pmod{2^{j+1}}. 
	\end{equation}
	As shown in the proof of Theorem \ref{mainthm16}, we have $ S_{\alpha}(z) \in M_{2^{j-1}}\left(\Gamma_{0}(9 \cdot 2^{\alpha +6}), \chi\right)$ where $\chi$ is as defined in Lemma \ref{lem11}. Applying Theorem \ref{thm2.4}, we get that there is an integer $c \geq 0$ such that for any $d \geq 1,$
	\begin{equation}\label{5e2}
		S_{\alpha}(z)|T_{p_1}|T_{p_2}\cdots |T_{p_{c+d}} \equiv 0\pmod{2^{d}}, 
	\end{equation}
	whenever the primes $p_1, p_2, \ldots p_{c+d} $ are coprime to $6.$
	Further, if $ p_1, p_2, \ldots p_{c+d} $ are distinct primes and if $n$ is coprime to $ p_1p_2 \cdots p_{c+d} ,$ then 
	\begin{equation}\label{5e3}
		C_{\alpha}(p_1p_2 \cdots p_{c+d} \cdot n) \equiv 0 \pmod{2^{d}}.
	\end{equation}
	Using \eqref{5e1}, \eqref{5e2} and \eqref{5e3} together, we conclude that
	$$ B_{2^{\alpha}} \left(\frac{p_1p_2\cdots p_{c+d} \cdot n +1 -2^{\alpha +1}}{24} \right) \equiv 0 \pmod{2^{d}}.$$ This completes the proof of Theorem \ref{mainthm3}.
	\qed
	\section{Proof of Theorem \ref{thm12}}
	\begin{proof}[Proof of Theorem \ref{thm12}] 
		From equation \eqref{eq501}, we have
		\begin{equation}\label{eq300}
			\sum_{n=0}^{\infty} B_{2}(n)q^{n} = \frac{(q^2;q^2)^2_{\infty}}{(q;q)^2_{\infty}}.
		\end{equation}
		Note that for any positive integer  $k$, we have
		\begin{equation}\label{eq300a*}
			(q^{2k};q^{2k})_{\infty} \equiv (q^k;q^k)^{2}_{\infty} \pmod 2.
		\end{equation}
		\begin{equation}\label{eq300a}
			(q^{2k};q^{2k})^2_{\infty} \equiv (q^k;q^k)^{4}_{\infty} \pmod 4 .
		\end{equation}
		From \eqref{eq300} and \eqref{eq300a}, we get
		\begin{equation}\label{eq300b}
			\sum_{n=0}^{\infty} B_{2}(n)q^{n} = \frac{(q^2;q^2)^2_{\infty}}{(q;q)^2_{\infty}} \equiv (q;q)^{2}_{\infty} \pmod 4 .
		\end{equation}
		Thus we have
		\begin{equation}\label{eq301}
			\sum_{n=0}^{\infty} B_{2}(n)q^{12n+1} \equiv q (q^{12};q^{12})^{2}_{\infty}  \equiv \eta^2(12z)\pmod 4 .
		\end{equation}
		Using Theorem \ref{thm2.3}, we obtain $\eta^2(12z) \in S_1\left(\Gamma_{0}(144), \left(\frac{12^2}{\bullet}\right)\right). $ Observe that $\eta^2(12z) $ has a Fourier expansion i.e.
		\begin{equation}\label{eq302}
			\eta^2(12z)= q-2q^{13}- q^{25}- \cdots= \sum_{n=1}^{\infty}  c(n) q^n. 
		\end{equation}
		Thus, $c(n)=0$ if $n \not \equiv 1 \pmod {12},$ for all $n \geq 0.$ From \eqref{eq301} and \eqref{eq302}, comparing the coefficient of $q^{12n+1}$, we get
		\begin{equation}\label{eq303}
			B_{2}(n) \equiv c(12n+1) \pmod 2.
		\end{equation}
		Since $ \eta^2(12z)$ is a Hecke eigenform (see \cite{Martin1996}), it gives
		$$\eta^2(12z)|T_p= \sum_{n=1}^{\infty} \left(c(pn)+  \left(\frac{12^2}{p}\right) c\left(\frac{n}{p}\right)\right) q^n = \lambda(p)\sum_{n=1}^{\infty} c(n)q^n.$$ Observe that $ \left(\frac{12^2}{p}\right) = 1 .$ Comparing the coefficients of $q^n$ on both sides of the above equation, we obtain
		\begin{equation}\label{eq304}
			c(pn)+ c\left(\frac{n}{p}\right) = \lambda(p) c(n).
		\end{equation}
		Since $c(1)=1$ and $c(\frac{1}{p})=0,$ if we put $n=1$ in the above expression, we get $c(p)=\lambda(p).$ As $c(p)=0$ for all $p \not \equiv 1 \pmod {12}$ this implies that $\lambda(p)=0$ for all $p \not \equiv 1 \pmod {12}.$ From \eqref{eq304} we get that for all $p \not \equiv 1 \pmod {12},$ we have  
		\begin{equation}\label{eq305}
			c(pn)+ c\left(\frac{n}{p}\right) =0.
		\end{equation}
		Now, we consider two cases here. If $p \not| n,$ then 
		replacing $n$ by $pn+r$ with $\gcd(r,p)=1$ in \eqref{eq305}, we get
		\begin{equation}\label{eq306}
			c(p^2n+ pr)=0 .
		\end{equation}
		Replacing $n$ by $12n-pr+1$ in \eqref{eq306} and using \eqref{eq303}, we get 
		\begin{equation}\label{eq307}
			B_{2}\left(p^2n +pr \frac{(1- p^2)}{12} +  \frac{p^2-1}{12}\right) \equiv 0 \pmod4.
		\end{equation}
		Since $p  \not \equiv 1 \pmod {12},$ we have $12| (1-p^2)$ and  $\gcd\left( \frac{(1- p^2)}{12} , p\right)=1,$ when $r$ runs over a residue system excluding the multiples of $p$, so does $ r \frac{(1- p^2)}{12}.$
		Thus for $p \nmid j,$ \eqref{eq307} can be written as
		\begin{equation}\label{eq308}
			B_{2}\left(p^2n+pj+ \frac{p^2-1}{12}\right) \equiv 0 \pmod4.
		\end{equation}
		Now we consider the second case, when $p |n.$ Here replacing $n$ by $pn$ in \eqref{eq305}, we get
		\begin{equation}\label{eq309}
			c(p^2n) =  - c\left(n\right).
		\end{equation}
		Further substituting $n$ by $12n+1$ in \eqref{eq309} we obtain
		\begin{equation}\label{eq310}
			c(12p^2n+ p^2) = -  c\left(12n+1\right).
		\end{equation}
		Using \eqref{eq303} in \eqref{eq310}, we get
		\begin{equation}\label{eq311}
			B_{2}\left(p^2n+ \frac{p^2-1}{12} \right) = (-1) \cdot B_{2} \left(n\right).
		\end{equation}
		Let $p_i$ be primes such that $p_i \not \equiv 1 \pmod 12.$ Further note that
		$$ p_1^2 p_2^2 \cdots p_k^2 n + \frac{p_1^2 p_2^2 \cdots p_k^2-1}{12}= p_1^2 \left(p_2^2 \cdots p_k^2 n + \frac{ p_2^2 \cdots p_k^2 -1}{12} \right)+ \frac{p_1^2-1}{12} .$$
		Repeatedly using \eqref{eq311} and \eqref{eq308}, we get
		\begin{align*}
			&	B_{2} \left( p_1^2 p_2^2 \cdots p_k^2 p_{k+1}^2 n + \frac{p_1^2 p_2^2 \cdots p_k^2 p_{k+1}\left(12j+ p_{k+1}\right)-1}{12} \right) \\
			& \equiv (-1) \cdot   B_{2} \left(p_2^2 \cdots p_k^2 p_{k+1}^2 n + \frac{ p_2^2 \cdots p_k^2 p_{k+1}\left(12j+ p_{k+1}\right) -1}{12} \right) \equiv \cdots\\
			& \equiv (-1)^k \cdot B_{2} \left(p_{k+1}^2 n + p_{k+1} j + \frac{p^2_{k+1}-1}{12} \right) \equiv 0 \pmod4
		\end{align*}
		when $j \not \equiv 0 \pmod{ p_{k+1}}.$ This completes the proof of the theorem.
	\end{proof}
	

	\section{Proof of Theorem \ref{thm13}}
	\begin{proof}[Proof of Theorem \ref{thm13}]  Let $t \in \{ 5, 7, 11\}$. From \eqref{eq305}, we get that for any prime $p \equiv t \pmod {12}$
		\begin{equation}\label{eq312}
			c(pn)= (-1) \cdot  c\left(\frac{n}{p}\right).
		\end{equation}  Replacing $n$ by $12n+t$ we obtain
		\begin{equation}\label{eq313}
			c(12pn+tp)= (-1) \cdot c \left(\frac{12n+t}{p}\right).
		\end{equation}
		Further replacing $n$ by $p^kn+r$ with $p \nmid r$ in \eqref{eq313}, we obtain
		\begin{equation}\label{eq314}
			c\left(12 \left(p^{k+1}n+pr+ \frac{tp-1}{12} \right)+1\right)= (-1) \cdot  c\left(12 \left(p^{k-1}n+ \frac{12r+t-p}{12p}\right)+1\right).
		\end{equation}
		Observe that $\frac{tp-1}{12} $ and $\frac{12r+t-p}{12p}$ are integers. Using \eqref{eq314} and \eqref{eq303}, we get
		\begin{align}\label{eq315}
			B_{2}\left(p^{k+1}n+pr+ \frac{tp-1}{12}\right) &\equiv  (-1) \cdot  B_{2}\left(p^{k-1}n+ \frac{12r+t-p}{12p} \right) \\ 
			&\equiv B_{2}\left(p^{k-1}n+ \frac{12r+t-p}{12p} \right)\pmod4.
		\end{align}
	\end{proof}
	
	\begin{proof}[Proof of Corollary \ref{coro8}]  Let $t \in \{ 5, 7, 11\}$. Let $p$ be a prime such that $p \equiv t \pmod{12}.$ Choose a non negative integer $r$ such that $12r+t=p^{2k-1}.$ Substituting $ k$ by $2k-1$ in \eqref{eq315}, we obtain
		\begin{align*}
			B_{2}\left(p^{2k}n+ \frac{p^{2k}-1}{12} \right)&\equiv  (-1)  B_{2}\left(p^{2k-2}n+ \frac{p^{2k-2}-1}{12} \right)\\
			& \equiv \cdots \equiv \left(-1\right)^k  B_{2}(n)  \equiv   B_{2}(n) \pmod4.
		\end{align*}
	\end{proof}
	
	\section{Proof of Theorem \ref{thm14}}
	\begin{proof}[Proof of Theorem \ref{thm14}] 
		From equation \eqref{eq501}, we have
		\begin{equation}\label{eq400}
			\sum_{n=0}^{\infty} B_{4}(n)q^{n} = \frac{(q^4;q^4)^2_{\infty}}{(q;q)^2_{\infty}}.
		\end{equation}
		Note that for any prime number $p$ and positive integer $j$, we have
		\begin{equation}\label{eq400a}
			(q;q)^{p^j}_{\infty} \equiv 	(q^p;q^p)^{p^{j-1}}_{\infty} \pmod {p^j}.
		\end{equation}
		Thus, by putting $p=2$ and $j=2,$ we get
		\begin{equation}\label{eq400aa}
			(q^2;q^2)^2_{\infty} \equiv (q;q)^{4}_{\infty} \pmod 4.
		\end{equation} Replacing $q$ by $q^2$  we get
		\begin{equation}\label{eq400aaab}
			(q^4;q^4)^2_{\infty} \equiv (q^2;q^2)^{4}_{\infty} \pmod 4.
		\end{equation}
		If we put $p=2$ and $j=3$ in \eqref{eq400a} we get
		\begin{equation}\label{eq400aaa}
			(q^2;q^2)^{4}_{\infty}	 \equiv (q;q)^8_{\infty}  \pmod 8.
		\end{equation}
		From \eqref{eq400} , \eqref{eq400aaab} and \eqref{eq400aaa}, we get
		\begin{equation}\label{eq400b}
			\sum_{n=0}^{\infty} B_{4}(n)q^{n} = \frac{(q^4;q^4)^2_{\infty}}{(q;q)^2_{\infty}} \equiv (q^2;q^2)^{3}_{\infty} \equiv (q;q)^{6}_{\infty} \pmod 4 .
		\end{equation}
		Thus, we have
		\begin{equation}\label{eq401}
			\sum_{n=0}^{\infty} B_{4}(n)q^{4n+1} \equiv q (q^4;q^4)^{6}_{\infty}  \equiv \eta^6(4z)\pmod  4 .
		\end{equation}

		By using Theorem \ref{thm2.3}, we obtain $\eta^6(4z) \in S_3(\Gamma_{0}(16), \left(\frac{4^6}{\bullet}\right)). $ Notice that $\eta^6(4z) $ has a Fourier expansion i.e.
		\begin{equation}\label{eq402}
			\eta^6(4z)= q-6q^5+9 q^{9}- \cdots= \sum_{n=1}^{\infty}  d(n) q^n.
		\end{equation}
		Thus, $d(n)=0$ if $n \not \equiv 1 \pmod 4,$ for all $n \geq 0.$ From \eqref{eq401} and \eqref{eq402}, comparing the coefficient of $q^{4n+1}$, we get
		\begin{equation}\label{eq403}
			B_{4}(n) \equiv d(4n+1) \pmod 4.
		\end{equation}
		Since $ \eta^6(4z)$ is a Hecke eigenform (see \cite{Martin1996}), it gives
		$$\eta^6(4z)|T_p= \sum_{n=1}^{\infty} \left(d(pn)+ p^2 \cdot  \left(\frac{4^6}{p}\right) d\left(\frac{n}{p}\right)\right) q^n = \lambda(p)\sum_{n=1}^{\infty} d(n)q^n.$$ Note that $ \left(\frac{4^6}{p}\right) = 1. $ Comparing the coefficient of $q^n$ on both sides of the above equation, we get
		\begin{equation}\label{eq404}
			d(pn)+ p^2 \cdot d\left(\frac{n}{p}\right) = \lambda(p) d(n).
		\end{equation}
		Since $d(1)=1$ and $d(\frac{1}{p})=0,$ if we put $n=1$ in the above expression, we get $d(p)=\lambda(p).$ As $d(p)=0$ for all $p \not \equiv 1 \pmod 4$ this implies that $\lambda(p)=0$ for all $p \not \equiv 1 \pmod 4.$ From \eqref{eq404} we get that for all $p \not \equiv 1 \pmod 4$  
		\begin{equation}\label{eq405}
			d(pn)+ p^2 \cdot d\left(\frac{n}{p}\right) =0.
		\end{equation}
		Now, we consider two cases here. If $p \not| n,$ then 
		replacing $n$ by $pn+r$ with $\gcd(r,p)=1$ in \eqref{eq405}, we get
		\begin{equation}\label{eq406}
			d(p^2n+ pr)=0 .
		\end{equation}
		Replacing $n$ by $4n-pr+1$ in \eqref{eq406} and using \eqref{eq403}, we get 
		\begin{equation}\label{eq407}
			B_{4}\left(p^2n +pr \frac{(1- p^2)}{4} +  \frac{p^2-1}{4}\right) \equiv 0 \pmod4.
		\end{equation}
		Since $p \equiv 3 \pmod 4,$ we have $4| (1-p^2)$ and  $\gcd\left( \frac{(1- p^2)}{4} , p\right)=1,$ when $r$ runs over a residue system excluding the multiples of $p$, so does $ r \frac{(1- p^2)}{4}.$
		Thus for $p \nmid j,$ \eqref{eq407} can be written as
		\begin{equation}\label{eq408}
			B_{4}\left(p^2n+pj+ \frac{p^2-1}{4}\right) \equiv 0 \pmod4.
		\end{equation}
		Now we consider the second case, when $p |n.$ Here replacing $n$ by $pn$ in \eqref{eq405}, we get
		\begin{equation}\label{eq409}
			d(p^2n) =  -p^2 \cdot d\left(n\right).
		\end{equation}
		Further substituting $n$ by $4n+1$ in \eqref{eq409} we obtain
		\begin{equation}\label{eq410}
			d(4p^2n+ p^2) = -p^2 \cdot  d\left(4n+1\right).
		\end{equation}
		Using \eqref{eq403} in \eqref{eq410}, we get
		\begin{equation}\label{eq411}
			B_{4}\left(p^2n+ \frac{p^2-1}{4} \right) = - p^2 \cdot B_{4} \left(n\right).
		\end{equation}
		Let $p_i$ be primes such that $p_i \equiv 3 \pmod 4.$ Further note that
		$$ p_1^2 p_2^2 \cdots p_k^2 n + \frac{p_1^2 p_2^2 \cdots p_k^2-1}{4}= p_1^2 \left(p_2^2 \cdots p_k^2 n + \frac{ p_2^2 \cdots p_k^2 -1}{4} \right)+ \frac{p_1^2-1}{4} .$$
		Repeatedly using \eqref{eq411} and \eqref{eq408}, we get
		\begin{align*}
			&	B_{4} \left( p_1^2 p_2^2 \cdots p_k^2 p_{k+1}^2 n + \frac{p_1^2 p_2^2 \cdots p_k^2 p_{k+1}\left(4j+ p_{k+1}\right)-1}{4} \right) \\
			& \equiv (- p_1^2) \cdot  B_{4} \left(p_2^2 \cdots p_k^2 p_{k+1}^2 n + \frac{ p_2^2 \cdots p_k^2 p_{k+1}\left(3j+ p_{k+1}\right) -1}{4} \right) \equiv \cdots\\
			& \equiv (-1)^k ( p_1^2 p_2^2 \cdots p_k^2)^k\cdot B_{4} \left(p_{k+1}^2 n + p_{k+1} j + \frac{p^2_{k+1}-1}{4} \right) \equiv 0 \pmod4
		\end{align*}
		when $j \not \equiv 0 \pmod{ p_{k+1}}.$ This completes the proof of the theorem.
	\end{proof}
	

	\section{Proof of Theorem \ref{thm15}}
	\begin{proof}[Proof of Theorem \ref{thm15}] From \eqref{eq405}, we get that for any prime $p \equiv 3 \pmod 4$
		\begin{equation}\label{eq412}
			d(pn)= -p^2 \cdot d\left(\frac{n}{p}\right).
		\end{equation}  Replacing $n$ by $4n+3,$ we obtain
		\begin{equation}\label{eq413}
			d(4pn+3p)= - p^2 \cdot d\left(\frac{4n+3}{p}\right).
		\end{equation}
		Further replacing $n$ by $p^kn+r$ with $p \nmid r$ in \eqref{eq413}, we obtain
		\begin{equation}\label{eq414}
			d\left(4 \left(p^{k+1}n+pr+ \frac{3p-1}{4} \right)+1\right)= (-p^2) \cdot  d \left(4\left(p^{k-1}n+ \frac{4r+3-p}{4p}\right)+1\right).
		\end{equation}
		Notice that $\frac{3p-1}{4} $ and $\frac{4r+3-p}{4p}$ are integers. Using \eqref{eq414} and \eqref{eq403}, we get
		\begin{equation}\label{eq415}
			B_{4}\left(p^{k+1}n+pr+ \frac{3p-1}{4}\right)\equiv  (-p^2) \cdot   B_{4}\left(p^{k-1}n+ \frac{4r+3-p}{4p} \right) \pmod4.
		\end{equation}
	\end{proof}
	
	\begin{proof}[Proof of Corollary \ref{coro10}]Let $p$ be a prime such that $p \equiv 3 \pmod 4.$ Choose a non negative integer $r$ such that $4r+3=p^{2k-1}.$ Substituting $ k$ by $2k-1$ in \eqref{eq415}, we obtain
		\begin{align*}
			B_{4}\left(p^{2k}n+ \frac{p^{2k}-1}{4} \right)&\equiv  (-p^2) \cdot B_{4}\left(p^{2k-2}n+ \frac{p^{2k-2}-1}{4} \right)\\
			& \equiv \cdots \equiv \left(-p^2\right)^k \cdot   B_{4}(n) \pmod4.
		\end{align*}
	\end{proof}

	%
	\noindent{\bf Data availability statement:} There is no data associated to our manuscript.
	
	


\begin{thebibliography}{99}
		\bibitem{Allouche2018} J. P. Allouche and L. Goldmakher, \emph{Mock character and Kronecker symbol,} J. Number Theory \textbf{192} (2018), 356-372.
		
		\bibitem{Andrews2010} G.E. Andrews, M.D. Hirschhorn and J. A. Sellers, \emph{Arithmetic properties of partitions with even parts distinct,} Ramanujan J. \textbf{23} (2010), 169-181.
		\bibitem{Aricheta2017} V. M. Aricheta, \emph{Congruences for Andrews' $(k,i)-$ singular overpartitions}, Ramanujan J. \textbf{43} (2017), 535-549.
		
		\bibitem{Barman2021} R. Barman and A. Singh, \emph{On mex-related partition functions of Andrews and Newman}, Res. Number Theory \textbf{7} (2021), Paper No. 53, 11 pp.
		
		
		
		
		
		
		\bibitem{Chen2013} S. Chen, \emph{Congruences for $t$-core partition functions}, J. Number Theory \textbf{133} (2013), 4036-4046.
		
		\bibitem{Dai2015} H. Dai, \emph{Congruences for the number of partitions and bipartitions with distinct even parts}, Discrete Math. \textbf{338} (2015), 133-138.
		
		\bibitem{Kathiravan2017}T. Kathiravan and S. N. Fathima, \emph{On $\ell$-regular bipartitions modulo $\ell$}, Ramanujan J. \textbf{44} (2017), 549-558.
		
		\bibitem{Lin2013}B. L. S. Lin, \emph{Arithmetic properties of bipartitions with even parts distinct}, Ramanujan J. \textbf{33} (2013), 269-279.
		\bibitem{Lin2015}B. L. S. Lin, \emph{Arithmetic of  the $7$-regular bipartition function modulo $3$}, Ramanujan J. \textbf{37} (2015), 469-478.
		
		\bibitem{Lin2016}B. L. S. Lin, \emph{An infinite family of congruences modulo $3$ for $13$-regular bipartitions}, Ramanujan J. \textbf{39} (2016), 169-178.
		
		\bibitem{Garvan1993} F. Garvan, \emph{Some congruences for partitions that are $p$-cores}, Proc. Lond. Math. Soc. \textbf{66} (1993), no. 3, 449-478.
		
		\bibitem{Garvan1990} F. Garvan, D. Kim and D. Stanton, \emph{Cranks and $t$-cores}, Invent. Math. \textbf{101} (1990), no. 1, 1-17.
		
		\bibitem{Gordan1997}B. Gordan and K. Ono, \emph{Divisibility of certain partition functions by power of primes}, Ramanujan j. \textbf{1} (1997), 25-34.
		
		
		
		\bibitem{Graville1996} A. Granville and K. Ono, \emph{Defect zero $p$-blocks for finite simple groups}, Trans. Amer. Math. Soc. \textbf{348} (1996),	no. 1, 331-347.
		
		
		\bibitem{Hirs2019} M. D. Hirschhorn and J. A. Sellers, \emph{Parity results for partitions wherein each parts an odd number of times}, Bull. Aust. Math. Soc. \textbf{1} (2019), 51-55.
		
		\bibitem{Koblitz} N. Koblitz, \emph{Introduction to elliptic curves and modular forms}, Springer-Verlag New York (1991).
		
		\bibitem{Martin1996} Y. Martin, \emph{Multiplicative $\eta$-quotients}, Trans. Am. Math. Soc. \textbf{348} (1996), 4825-4856.
		
		\bibitem{MeherJindal2022} N. K. Meher and A. Jindal, \emph{Arithmetic density and new congruences for $3$-core partitions}, Preprint.
		
		\bibitem{Newmann1959} M. Newmann, \emph{Modular forms whose coefficients possess multiplicative properties}, Ann. of Math. \textbf{70} (1959), 478--489.
		
		
		
		\bibitem{Ono2004} K. Ono, \emph{The web of modularity: arithmetic of the coefficients of modular forms and $q-$series,} CBMS Regional Conference Series in Mathematics, $102,$ Amer. Math. Soc., Providence, RI, 2004.
		
		\bibitem{Taguchi2005} K. Ono and Y. Taguchi, \emph{$2$-adic properties of certain modular forms and their application to arithmetic functions}, Int. J. Number Theory \textbf{1} (2005), 75-101.
		
		
		
		\bibitem{radu2009} S. Radu, \emph{An algorithmic approach to Ramanujan's congruences},   Ramanujan J. \textbf{20} (2009), 215-251.
		
		\bibitem{Radu2011} S. Radu and J. A. Sellers \emph{Congruences properties modulo $5$ and $7$ for the pod function, }   Int. J. Number
		Theory \textbf{7} (2011), 2249-2259.
		
		\bibitem{radu2011a} S. Radu and J. A. Sellers, \emph{Parity results for broken $k$-diamond partitions and $(2k+1)$-cores}, Acta Arith. \textbf{146} (2011), 43-52.
		
		
		\bibitem{Sage} The Sage Developers, \emph{sagemath, the Sage Mathematics Software System (Version 8.1).} https://www.sagemath.org
		
		\bibitem{Serre1974} J. -P. Serre, \emph{Divisibilit$\acute{e}$ des coefficients des formes modularies de poids entier,} C. R. Acad. Sci. Paris (A), \textbf{279} (1974), 679-682.
		
		\bibitem{Serre1974a} J. -P. Serre, \emph{Divisibilit$\acute{e}$ de certaines fonctions arithm$\acute{e}$tiques}, in: S$\acute{e}$minaire Delanga-Pisot-Poitou, Th$\acute{e}$orie Nr., \textbf{16} (1974), 1-28.
		
		\bibitem{wang2017} L. Wang, \emph{Arithmatic properties of $(k,l)$-regular bipartitions}, Bull. Aust. Math. Soc. \textbf{95} (2017), 353-364.
		
		\bibitem{Ajit2021} A. Singh and R. Barman, \emph{Certain eta-quotients and arithmetic density of Andrews' singular overpartitions}, J. Number Theory \textbf{229} (2021), 487-498.
		
		\bibitem{Tate1994} J. Tate, \emph{Extensions of $\mathbb{Q}$ un-ramified outside $2$}, in: Arithmetic Geometry: Conference on Arithmetic Geometry with an Emphasis on Iwasawa Theory, Arizona State University, March $15-18,$ 1993, Vol. 174, No. 174, American Mathematical Society, Providence, 1994.
	\end{thebibliography}
\end{document}